\documentclass[final,3p,times]{elsarticle}
\usepackage{amsmath, amssymb}
\usepackage{float}
\usepackage{balance}
\usepackage{amsfonts}
\usepackage{amsthm}
\usepackage{natbib}
\usepackage{epstopdf}
\usepackage{balance}
\usepackage{cases}
\usepackage{multirow}
\usepackage{color,soul}
\usepackage{booktabs}
\usepackage{changepage}

\usepackage[draft=false,
colorlinks=true,
linkcolor=blue,
citecolor=red,
urlcolor=green]{hyperref}
\usepackage{multirow}
\usepackage{lineno}
\usepackage{mathrsfs}
\theoremstyle{plain}
\newtheorem{theorem}{Theorem}[section]

\newtheorem{lemma}{Lemma}[section]

\usepackage{color}
\theoremstyle{remark}

\makeatletter    
\@addtoreset{equation}{section}
\makeatother    

\biboptions{numbers,sort&compress}

\journal{Journal of \LaTeX\ Templates}

\begin{document}

\begin{frontmatter}

\title{A  new alternating direction  trust region method based on conic model for solving unconstrained optimization }

\author[mymainaddress,mysecondaryaddress,mythirdaddress]{ Honglan Zhu}
\ead{zhuhonglan1981@163.com}

\author[mymainaddress]{ Qin Ni \corref{mycorrespondingauthor}}
\cortext[mycorrespondingauthor]{Corresponding author}
\ead{niqfs@nuaa.edu.cn}

\author[mythirdaddress]{Chuangyin Dang}
\ead{mecdang@cityu.edu.hk}


\address[mymainaddress]{Department of Mathematics, Nanjing University of Aeronautics and Astronautics, Nanjing 210016, People's Republic of China.}
\address[mysecondaryaddress]{Business School, Huaiyin Institute of Technology,  Huaian  223003, People's Republic of China.}
\address[mythirdaddress]{Department of Systems Engineering \& Engineering Management, City University of Hong Kong, Kowloon, Hong Kong SAR.}

\begin{abstract}
 In this paper, a new alternating direction trust region method based on conic model is used to solve unconstrained optimization problems.
By use of the alternating direction method,  the new conic model trust region subproblem is solved by two steps in two orthogonal directions.
This new idea overcomes the shortcomings of conic model subproblem which is difficult to solve.
Then the global convergence of the method under some reasonable conditions is established. Numerical experiment shows that this method may be better than  the dogleg method  to solve the subproblem, especially for large-scale problems.
\end{abstract}


\end{frontmatter}

\section{Introduction}

In this paper, we consider the unconstrained optimization problem
\begin{eqnarray}
 \min_{x\in R^n} \ \ f(x), \label{f}
\end{eqnarray}
where   $f(x)$ is continuously differentiable. The problem \eqref{f} have been studied by many researchers, including Han \cite{Han77}, Powell  \cite{Powell83}, Yuan and Sun \cite{yuansun97}, Powell and Yuan \cite{py90}, etc.
There are many  methods to solve problem \eqref{f}, and trust region method is a very effective method (see \cite{py90, Vardi1981, Boggs1987, TOINT1988,zz90,El-Alem}).
In addition, the book of Conn, Gould and Toint \cite{Conn2000} is an excellent and comprehensive one on trust region methods.
Most optimization theory is based on the quadratic model and  {\color{red}uses} the quadratic model to approximate $f(x)$.
That is, at the $k$th iteration, the following subproblem:
\begin{eqnarray}
&&  \min_{s \in R^n} \ \varrho_k(s)=g_k^Ts + \frac{1}{2} s^TB_ks, \label{quadratic}\\
&& \mbox{s.t.}  \ \ \ \ \| s \| \le  \Delta _k, \label{sDelta1}
\end{eqnarray}
is solved to obtain a search direction $s_k$, where $x_k$ is  the current iterate point,  $g_k=\nabla f(x_k)$, $B_k$ is symmetric and an approximation to the Hessian of $f(x)$, $\|\cdot\|$ refers to the Euclidean norm, $ \Delta_k$ is the trust region radius at the $k$th iteration.

There are many methods can be used to solve the subproblem \eqref{quadratic}-\eqref{sDelta1}.
The simple, low cost and effective methods are dogleg methods, such as Powell's single dogleg method  \cite{powell1970} and
Dennis and Mei's double dogleg method \cite{Dennis1979}. Then there are other scholars have studied the dogleg method \cite{ZT2001,ZXZ1987,Zhao2000}.
Now, we recall the  simple dogleg algorithm for solving trust region subproblem with the quadratic model as following algorithm.

\

\textbf{Algorithm 1.1}

Step 0.\quad Input the  data of the $k$th iteration i.e., $g_k, B_k$ and $\Delta_k$.

Step 1.\quad Compute $s_k^{\textrm{N}}=-B_k^{-1}g_k$. If $\|s_k^{\textrm{N}}\| \leq {\Delta}_k$, then $s_{\ast}=s_k^{\textrm{N}}$, and stop.

Step 2.\quad Compute $s_k^c=-\frac{g_k^Tg_k}{g_k^TB_kg_k}g_k$. 
If $\|s_k^c\| \geq {\Delta}_k$, then $s_{\ast}=- \frac{{\Delta}_kg_k}{\|g_k\|}$, and stop. Otherwise, go to Step 3.

Step 3.\quad Compute
\begin{eqnarray}\label{def}
  d =  \|s_k^{\textrm{N}}-s_k^{c}\|^2, \ \ \ e = (s_k^{\textrm{N}}-s_k^{c})^{T}s_k^{c}, \ \ \ f =\|s_k^{c}\|^2-{\Delta}_k^2,
\end{eqnarray}
\qquad \qquad \quad \, then $s_{\ast}= s_k^{c}+\lambda(s_k^{\textrm{N}}-s_k^{c})$, where $\lambda =  \frac{-e+\sqrt{e^2-df}}{d}$.

\

We note that the solution of the subproblem obtained by dogleg methods  is  only an approximate solution of  \eqref{quadratic}-\eqref{sDelta1}.
Moreover, practice experience shows that the quadratic model is not always effective. If the objective function possesses high non-linear property and the iterative point is far away from the minimum,
the quadratic model could not approximate the original problem very well, which may lead to iteration proceed slowly.

In 1980, Davidon \cite{Davidon1980} proposed the conic model for solving unconstrained optimization.
It is an alternative model to substitute the quadratic model. And it has attracted wide attention of many authors in various areas \cite{yu00, ga81, py97, sc82, so82, sy2001, xy98}.
 A typical trust-region subproblem with conic model was first
 proposed by Di and Sun in \cite{ds96}  as following.
\begin{eqnarray}
&&   \min_{s \in R^n}   \   \displaystyle \phi_k (s)= \frac{g_k^Ts}{1- a_k^Ts} + \frac{ s^TB_ks}{2(1-a_k^Ts)^2},\label{phi3}\\
&& \mbox{s.t.}   \ \ \ \  \| s \| \le \Delta _k, 1-a_k^Ts>0, \label{s1a}
  \end{eqnarray}
where {\color{red}horizon vector $a_k\in R^n$,} and $B_k$ is symmetric and positive semidefinite.
 In \cite{ni05}, Ni proposed a new trust region
 subproblem and gave the optimality conditions for the trust region subproblems of a conic model.
 That is,  at the $k$th iteration, the trial step $s_k$ is computed by solving the following conic model trust region subproblem
\begin{eqnarray}
&&  \min_{s \in R^n} \  \phi_k (s) = \frac{g_k^Ts}{1 - a_k^Ts} +  \frac{s^TB_ks}{2(1 - a_k^Ts)^2},\label{phis} \\
&& \mbox{s.t.} \ \ \ \  \| s \| \le  \Delta _k, \ \  |1 - a_k^Ts|\geq \varepsilon_0, \label{sDelta}
\end{eqnarray}
where $\varepsilon_0$ ($0<\varepsilon_0<1$) is a sufficiently small positive number. The  subproblem \eqref{phis}-\eqref{sDelta} considered more comprehensive than \eqref{phi3}-\eqref{s1a}, and will not  miss the solution of  the original problem  \eqref{f}.

The research demonstrated that the conic model is superior to quadratic model to some extent, in particular, for those class of objective
functions with highly vibrating; in addition, the conic model can supply enough freedom to make best use of
both information of gradients and function values in iterate points. In view of this good properties  of conic model, we will continue to study it.

It is noteworthy that the  simple dogleg algorithm for solving trust region subproblem based on the conic model (DCTR) is similar to the above  Algorithm 1.1, where
\begin{eqnarray*}
&&  s_k^{\textrm{N}}=\frac{-B_k^{-1}g_k}{1-a_k^TB_k^{-1}g_k}, \\
&& s_k^c=\frac{-g_k^Tg_k}{g_k^TB_kg_k-a_k^Tg_kg_k^Tg_k}g_.
\end{eqnarray*}
 However, the calculation of DCTR is much more complicated (see \cite{zxz95,lu08,ZhaoS13})

In order to find a simpler method and which is  more suitable for the unique structure of the conic model, we considered  to using the alternating directions method for solving the conic model subproblem.
Alternating directions method (ADM) could date back to \cite{Gabay76}. It has been well studied in the  linearly
constrained convex programming problems.  Because of its significant efficiency and easy implementation, ADM
has attracted wide attention of many authors in various areas, see  \cite{Chen1994,Eckstein1994,HeLY2002,KonM1998,zhangLS17,Xu07}.

In this paper, we combine the subproblem \eqref{phis}-\eqref{sDelta} with alternating direction search method to  propose a new method for solving the conic trust region subproblem. The rest of this paper is organized as follows. In the next section, the motivation and description of the simple alternating direction search algorithm are presented.  In  Section 3, we give the quasi-Newton method based on the conic model for solving unconstrained
optimization problems and prove its global convergence properties. The numerical results  in Section 4 indicate that the algorithm is efficient and robust.

\section{A simple alternating direction search method}

 The conic model $\phi_k (s)$ in the subproblem \eqref{phis}-\eqref{sDelta} has one more   parameter $a_k$ than $\varrho_k(s)$, so $\phi_k (s)$ has more freedom which can take into account the information concerning the function value in the previous iteration which is useful for algorithms.
Furthermore,  the conic model possesses richer interpolation information and can satisfy four interpolation conditions of the
function values and the gradient values at the current and the previous points.
Using these rich interpolation information may improve the performance of the algorithms.
Generally, the choice of the parameters $a_k$ is a descent direction, such as  $g(x_{k-1})$, $g(x_k)$ or $s_{k-1}$ (see \cite{Davidon1980,sc82,so82,lu08,zxz95}).

In view of the unique importance of the parameters $a_k$,   we  consider the following  alternating direction search method to solve the subproblem \eqref{phis}-\eqref{sDelta}.  The new method is divided into two steps. First, we search along the  direction parallel to $a_k$. And then search along the  direction $y_k$ which is perpendicular   to $a_k$.  For convenience, we omit the index $k$ of $a_k, g_k$ and $B_k$ in this section.

In this paper, we assume that $a\neq0$ and $B$ is positive (abbreviated as $B> 0$).

Let
\begin{eqnarray}\label{stay}
s=\tau a +y,
\end{eqnarray}
where $\tau \in R,  y \in R^n$ and $a^T y =0$. Then, the solving process  of subproblem \eqref{phis}-\eqref{sDelta} is divided into the following two stages.

In the first stage, we set $y =0$ and then $s=\tau a$. Substituting it into \eqref{phis}-\eqref{sDelta}, we have
\begin{eqnarray}
 && \min  \ \ \ \rho (\tau)= \frac{\tau a^Tg}{1- \tau a^Ta} + \frac{ \tau^2 a^TBa}{2(1- \tau a^Ta)^2}, \label{rhotau}\\
&& \mbox{s.t.} \ \ \ \ \ \tau \in \Omega ,  \label{tauome}
\end{eqnarray}
where $\Omega= \{\tau \ | \ | \tau | \| a \|\leq \Delta, | 1- \tau  \| a \|^2 |\geq \varepsilon_0 \}$.

For the purpose of clarity, we denote
\begin{eqnarray}\label{tauDdu}
  \tau_{ \Delta}= \frac{\Delta}{\| a \|}, \ \ \tau_d= \frac{1- \varepsilon_0}{\| a \|^2}, \ \ \tau_m = \frac{1}{\| a \|^2},  \ \ \tau_u= \frac{1+ \varepsilon_0}{\| a \|^2}.
\end{eqnarray}
Then,
 \begin{eqnarray}\label{Omega}
\Omega= \{\tau \ | \ | \tau | \leq \tau_{ \Delta}\}\cap\{ \tau \ |  \tau\leq \tau_d \ or  \ \tau \geq \tau_u\}.
\end{eqnarray}

 In the following, we consider three different cases of \eqref{rhotau}-\eqref{tauome}:

 (1) If $\Delta \|a\|\leq 1-\varepsilon_0$, then $\tau_{ \Delta}\leq\tau_d$ and \eqref{rhotau}-\eqref{tauome} becomes
\begin{eqnarray}
 (\text{P1}) \left\{\begin{array}{l}
  \min  \ \ \ \rho (\tau), \\
 \mbox{s.t.} \ \ \ \ \   \tau \in [-\tau_{ \Delta}, \tau_{ \Delta}]. \ \ \ \ \ \ \ \ \ \ \ \ \ \ \ \ \ \end{array}\right.
\end{eqnarray}

(2) If $|1- \Delta \|a\| \, |<  \varepsilon_0$, then  $\tau_d<\tau_{ \Delta}<\tau_u$ and \eqref{rhotau}-\eqref{tauome} becomes
 \begin{eqnarray}
 (\text{P2}) \left\{\begin{array}{l}
 \min  \ \ \ \rho (\tau), \\
 \mbox{s.t.} \ \ \ \ \   \tau \in [-\tau_{\Delta}, \tau_d]. \ \ \  \ \ \ \ \ \ \ \ \ \  \  \ \  \ \end{array}\right.
\end{eqnarray}

(3) If $\Delta \|a\| \geq 1+\varepsilon_0$, then $\tau_u\leq\tau_{ \Delta}$ and \eqref{rhotau}-\eqref{tauome} becomes
 \begin{eqnarray}
 (\text{P3}) \left\{\begin{array}{l}
  \min  \ \ \ \rho (\tau), \\
 \mbox{s.t.} \ \ \ \ \   \tau \in [-\tau_{\Delta}, \tau_d]\cup [\tau_u,  \tau_{\Delta}].\end{array}\right.
\end{eqnarray}

Now, we discuss the stationary points of $\rho (\tau)$. By the direct computation,
 we have that the   derivative of  $\rho (\tau)$ is 
{\color{red}\begin{eqnarray}\label{rhodao}
\rho '(\tau)=  \frac{ a_{\tau}\tau + a^Tg}{-\| a \|^6(\tau- \tau_m)^3},
 \end{eqnarray}}
 where
\begin{eqnarray}\label{ataum}
a_{\tau} = a^TBa-a^Taa^Tg.
 \end{eqnarray}
From \eqref{tauDdu}, we know that $0<\tau_d<\tau_m<\tau_u$ and {\color{red}then  from \eqref{Omega}} $\tau_m\not\in \Omega$.
 Therefore, if $a_{\tau}\neq 0$ then $\rho(\tau)$ has only one   stationary point
\begin{eqnarray}\label{tau1}
\tau_{cp}=  \frac{ -a^Tg}{a_{\tau}}.
 \end{eqnarray}


 \begin{lemma} (1) \  If $a_{\tau}<0$ then $\tau_m < \tau_{cp}$
 and $\rho (\tau)$ is  monotonically decreasing in the in the trust region {\color{red} $(\tau_m,  \tau_{cp})$};
 $\rho (\tau)$ is  monotonically increasing  for  {\color{red}$ \tau < \tau_m$ and    $\tau > \tau_{cp}$}.

 (2) \  If  $a_{\tau} = 0$, then $a^Tg>0$ and $\rho (\tau)$ is  monotonically increasing for  {\color{red}$\tau < \tau_m$};
 $\rho (\tau)$ is  monotonically decreasing for  {\color{red}$ \tau > \tau_m$}.

  (3) \  If  $a_{\tau} > 0$, then $\tau_{cp}< \tau_m$ and $\rho (\tau)$ is  monotonically increasing in  the trust region  {\color{red}$(\tau_{cp},  \tau_m)$};
 $\rho (\tau)$ is  monotonically decreasing for  {\color{red}$ \tau < \tau_{cp}$} and    {\color{red}$\tau > \tau_m$.}
  \end{lemma}

 \begin{proof} From \eqref{tauDdu} and \eqref{tau1}, we know that if   $a_{\tau} \neq 0$ then
  \begin{eqnarray}\label{tau1m}
\tau_{cp}-\tau_m =  \frac{a^TBa}{- a_{\tau}\| a \|^2}.
 \end{eqnarray}
 Then, {\color{red}since $B \succ 0$,} combining with \eqref{rhodao} we can obtain that the lemma obviously holds.
\end{proof}

  \begin{theorem}   \  If $a^Tg = 0$ then the optimal solution of the subproblem (\text{P1}), (\text{P2}) and (\text{P3}) is
 \begin{equation}\label{tauast0}
\tau_{\ast} = 0.
\end{equation}
  \end{theorem}

 \begin{proof} If $a^Tg = 0$ then from \eqref{rhotau} we have
 \begin{eqnarray}
\rho (\tau)= \frac{ \tau^2 a^TBa}{2(1- \tau a^Ta)^2} \geq 0.
\end{eqnarray}
Hence,   the theorem  holds.
\end{proof}

 \begin{theorem}   \  If  $a^Tg \neq 0$, then the optimal solution of the subproblem  (\text{P1}) is
 \begin{equation}\label{tauastP1}
\tau_{\ast}=\left\{\begin{array}{l}
-\tau_{\Delta}, \ \ \ \ \ \ \ \ \ \ \  \ \ \ \ \ \ \ \mbox{if} \ \ a_{\tau} \leq 0,\\
\max\{-\tau_{\Delta}, \tau_{cp}\}, \ \   \mbox{if} \ \ a_{\tau} > 0, a^Tg>0,\\
\min\{\tau_{cp}, \tau_{\Delta}\}, \ \ \ \ \     \mbox{if} \ \ a_{\tau} > 0, a^Tg<0.
\end{array}\right.
\end{equation}
  \end{theorem}

 \begin{proof} For  the subproblem  (\text{P1}), we know that $\Omega = [-\tau_{ \Delta}, \tau_{ \Delta}]$ where $\tau_{ \Delta}\leq\tau_d < \tau_m$.

 (1) If  $a_{\tau} \leq 0$, then from Lemma 2.1 (1)(2) we can easily obtain  $\tau_{\ast}= -\tau_{\Delta}$.

 (2) If $ a_{\tau} > 0, a^Tg>0$, then $\tau_{cp}<0$. From Lemma 2.1 (3), we can obtain that if $\tau_{cp}\leq -\tau_{\Delta}$ then $\tau_{\ast}= -\tau_{\Delta}$;
  If $-\tau_{\Delta} < \tau_{cp}<0$,  then  $\tau_{\ast}= \tau_{cp}$. Therefore,  $\tau_{\ast}= \max\{-\tau_{\Delta}, \tau_{cp}\}$.

  (3) If $ a_{\tau} > 0, a^Tg<0$, then $\tau_{cp}> 0$.  From Lemma 2.1 (3), we can obtain that if $0<\tau_{cp}\leq \tau_{\Delta}$ then $\tau_{\ast}= \tau_{cp}$;
  If $\tau_{\Delta} < \tau_{cp}<\tau_m$,  then  $\tau_{\ast}= \tau_{\Delta}$. Therefore,  $\tau_{\ast}= \min\{\tau_{cp}, \tau_{\Delta}\}$.
 \end{proof}

  \begin{theorem}   \  If  $a^Tg \neq 0$, then the optimal solution of the subproblem  (\text{P2}) is
 \begin{equation}\label{tauastP2}
\tau_{\ast}=\left\{\begin{array}{l}
-\tau_{\Delta}, \ \ \ \ \ \ \ \ \ \ \ \ \ \ \ \ \ \ \mbox{if} \ \ a_{\tau} \leq 0,\\
\max\{-\tau_{\Delta}, \tau_{cp}\}, \ \,   \mbox{if} \ \ a_{\tau} > 0, a^Tg>0,\\
\min\{\tau_{cp}, \tau_d\}, \ \ \ \ \     \mbox{if} \ \ a_{\tau} > 0, a^Tg<0.
\end{array}\right.
\end{equation}
  \end{theorem}

 \begin{proof} The proof process is similar to the above Theorem 2.2, so we omitted it.
%
 \end{proof}

  \begin{theorem}   \  If $ a_{\tau} < 0$, then $a^Tg>0$, $\tau_m < \tau_{cp}$ and the optimal solution of the subproblem  (\text{P3}) is
 \begin{equation}\label{tauastP3-1}
\tau_{\ast}=\left\{\begin{array}{l}
\tau_u, \ \  \  \mbox{if} \ \ \tau_m < \tau_{cp}\leq \tau_u,\\
\tau_{cp}, \   \mbox{if} \ \ \tau_u < \tau_{cp}< \tau_{\Delta},\\
\tilde{\tau}_{\Delta}, \ \   \mbox{if} \ \ \tau_{cp}\geq \tau_{\Delta},
\end{array}\right.
\end{equation}
where
\begin{eqnarray}\label{t-tauastD}
\tilde{\tau}_{\Delta} = \mbox{arg\:min}\{\rho (-\tau_{\Delta}), \rho (\tau_{\Delta})\}.
\end{eqnarray}
  \end{theorem}

 \begin{proof} For  the subproblem  (\text{P3}), we know that
\begin{eqnarray} \label{Omega-3}
 \Omega = [-\tau_{\Delta}, \tau_d]\cup [\tau_u,  \tau_{\Delta}],
 \end{eqnarray}
 where $\tau_d < \tau_m < \tau_u$.
If $ a_{\tau} < 0$, then $a^Tg>0$. And from Lemma 2.1 (1) we can easily obtain that $\tau_m < \tau_{cp}$ and
\begin{equation}\label{tauastP3a}
\tau_{\ast}=\left\{\begin{array}{l}
\mbox{arg\:min}\{\rho (-\tau_{\Delta}), \rho (\tau_u)\}, \ \    \mbox{if} \ \ \tau_m < \tau_{cp}\leq \tau_u,\\
\mbox{arg\:min}\{\rho (-\tau_{\Delta}), \rho (\tau_{cp})\}, \   \mbox{if} \ \ \tau_u < \tau_{cp}< \tau_{\Delta},\\
\mbox{arg\:min}\{\rho (-\tau_{\Delta}), \rho (\tau_{\Delta})\}, \ \   \mbox{if} \ \ \tau_{cp}\geq \tau_{\Delta}.
\end{array}\right.
\end{equation}

(1) If $\tau_m < \tau_{cp}\leq \tau_u$, then from  \eqref{rhotau} we have
\begin{eqnarray} \label{tauu-D}
\rho (\tau_u)- \rho (-\tau_\Delta)
= \frac{\Delta^2\|a\|^2a_{\Delta}+2\Delta \|a\|b_{\Delta} + c_{\Delta}}{2\varepsilon_0^2\|a\|^4(1+ \Delta \|a\|)^2 },
\end{eqnarray}
 where
 \begin{eqnarray}
 && a_{\Delta} = (1+2\varepsilon_0) a^TBa -2\varepsilon_0\|a\|^2 a^Tg, \\
 && b_{\Delta} = (1+\varepsilon_0)^2 a^TBa -(2+\varepsilon_0)\varepsilon_0\|a\|^2 a^Tg, \\
 && c_{\Delta} = (1+\varepsilon_0)^2 a^TBa -2(1+\varepsilon_0)\varepsilon_0\|a\|^2 a^Tg.
 \end{eqnarray}
 Because $\tau_{cp}\leq \tau_u$, then from \eqref{tauDdu} and \eqref{tau1} we have
  \begin{eqnarray}
- \varepsilon_0\|a\|^2 a^Tg  \leq  -(1+\varepsilon_0) a^TBa.
 \end{eqnarray}
 And then
  \begin{eqnarray}
 && a_{\Delta} \leq - a^TBa <0, \\
 && b_{\Delta} \leq -(1+\varepsilon_0) a^TBa <0, \\
 &&  c_{\Delta} \leq -(1+\varepsilon_0)^2 a^TBa <0.
 \end{eqnarray}
Combining with \eqref{tauu-D}, then
\begin{eqnarray*}
\rho (\tau_u)< \rho (-\tau_\Delta).
\end{eqnarray*}
Hence, $\tau_{\ast} = \tau_u$.

(2) If $\tau_u < \tau_{cp}< \tau_{\Delta}$, then from  \eqref{rhotau} we have
\begin{eqnarray} \label{tau1-D}
   \rho (\tau_{cp})- \rho (-\tau_\Delta)  
   = -\frac{ a_{\tau}^2\Delta^2 -2 a_{\tau}a^Tg\|a\|\Delta + \|a\|^2(a^Tg)^2} {2\|a\|^2(1+ \Delta \|a\|)^2 a^TBa}.
\end{eqnarray}
Because $a_{\tau}<0, a^Tg>0$, then
\begin{eqnarray*}
\rho (\tau_{cp})< \rho (-\tau_\Delta).
\end{eqnarray*}
Therefore, $\tau_{\ast} = \tau_{cp}$.
The theorem is proved.
 \end{proof}

 \begin{theorem}   \  If  $ a_{\tau} \geq 0$  and $a^Tg\neq0$, then the optimal solution of the subproblem  (\text{P3}) is \begin{equation}\label{tauastP3-2}
\tau_{\ast}=\left\{\begin{array}{l}
-\tau_{\Delta}, \ \ \ \ \ \ \ \ \ \ \  \ \ \ \ \ \ \ \mbox{if} \ \ a_{\tau} = 0,\\
\max\{-\tau_{\Delta}, \tau_{cp}\}, \ \,   \mbox{if} \ \ a_{\tau} > 0, a^Tg>0,\\
\min\{\tau_{cp}, \tau_d\}, \ \ \ \ \     \mbox{if} \ \ a_{\tau} > 0, a^Tg<0.
\end{array}\right.
\end{equation}
  \end{theorem}
 \begin{proof}
(1) If  $ a_{\tau} = 0$  then $a^Tg >0$. Combining \eqref{Omega-3} and Lemma 2.1 (2), we know that
  \begin{equation}\label{tauastP3-21}
\tau_{\ast}= \mbox{arg\:min}\{\rho (-\tau_{\Delta}), \rho (\tau_{\Delta})\}.
\end{equation}
However, by calculation we have
\begin{eqnarray}
 && \rho (\tau_\Delta)- \rho (-\tau_\Delta) \nonumber\\
 && = \frac{2\Delta a^Tg}{\|a\|(1- \Delta^2 \|a\|^2)}+ \frac{ 2\Delta^3 a^TBa}{\|a\|(1- \Delta^2 \|a\|^2)^2} \label{tauDD3}\\
 && = \frac{ 2\Delta(\Delta^2a_{\tau}+ a^Tg)}{\|a\|(1- \Delta^2 \|a\|^2)^2}.\label{tauDD}
\end{eqnarray}
For $ a_{\tau} = 0$  and $a^Tg>0$, then
\begin{eqnarray*}
\rho (\tau_\Delta)>\rho (-\tau_\Delta).
\end{eqnarray*}
 Hence, $\tau_{\ast} = -\tau_{\Delta}$ and \eqref{tauastP3-2} holds.

(2) If $a_{\tau} > 0, a^Tg>0$ then $\tau_{cp}<0$. Combining \eqref{Omega-3} and Lemma 2.1 (3), we know that the optimal solution of the subproblem  (\text{P3}) is
\begin{equation}\label{tauastP3-4}
\tau_{\ast}=\left\{\begin{array}{l}
\mbox{arg\:min}\{\rho (-\tau_{\Delta}), \rho (\tau_{\Delta})\}, \ \    \mbox{if} \ \   \tau_{cp}\leq -\tau_{\Delta},\\
\mbox{arg\:min}\{\rho (\tau_{cp}), \rho (\tau_{\Delta})\}, \ \ \ \ \mbox{if} \ \ -\tau_{\Delta} < \tau_{cp}< 0.
\end{array}\right.
\end{equation}
For $a_{\tau} > 0, a^Tg>0$, then from \eqref{tauDD} we note that
\begin{eqnarray}\label{tDD1}
\rho (\tau_\Delta)>\rho (-\tau_\Delta).
\end{eqnarray}
If $-\tau_{\Delta} < \tau_{cp}< 0$, then from Lemma 2.1 (3) we know that
\begin{eqnarray*}
\rho (-\tau_\Delta)>\rho (\tau_{cp}).
\end{eqnarray*}
Thus,
\begin{equation}\label{tauastP3-4}
\tau_{\ast}=\left\{\begin{array}{l}
-\tau_{\Delta}, \ \    \mbox{if} \ \   \tau_{cp}\leq -\tau_{\Delta},\\
\tau_{cp}, \ \ \ \ \mbox{if} \ \ -\tau_{\Delta} < \tau_{cp}< 0.
\end{array}\right.
\end{equation}
Then, \eqref{tauastP3-2} holds.

(3) If  $ a_{\tau} > 0$, $a^Tg<0$, then from \eqref{tau1} and \eqref{tau1m} we can get $0 <\tau_{cp}< \tau_m$.
 Combining \eqref{Omega-3} and Lemma 2.1 (3), we know that the optimal solution of the subproblem  (\text{P3}) is
\begin{equation}\label{tauastP3-5}
\tau_{\ast}=\left\{\begin{array}{l}
\mbox{arg\:min}\{\rho(\tau_{cp}), \rho (\tau_{\Delta})\}, \ \   \mbox{if} \ \   0 < \tau_{cp}< \tau_d,\\
\mbox{arg\:min}\{\rho (\tau_d), \rho (\tau_{\Delta})\}, \ \ \ \ \mbox{if} \ \ \tau_d \leq \tau_{cp}< \tau_m.
\end{array}\right.
\end{equation}
 For the subproblem  (\text{P3}),  we note that $1-\Delta \|a\| \leq -\varepsilon_0$. Because of  $a^Tg<0$, then
\begin{eqnarray} \label{tauDD6}
\rho (\tau_\Delta)= \frac{\Delta a^Tg}{\|a\|(1- \Delta \|a\|)} + \frac{ \Delta^2 a^TBa}{2\|a\|^2(1- \Delta \|a\|)^2}>0.
\end{eqnarray}
However, {\color{red}from $\rho(0) = 0$} and Lemma 2.1 (3) we can obtain  that if $0 < \tau_{cp}< \tau_d$ then $\rho (\tau_{cp})< 0$;
If $\tau_d \leq \tau_{cp}< \tau_m$ then $\rho (\tau_d)< 0$ holds too.
Therefore, it follows that
\begin{equation}\label{tauastP3-5}
\tau_{\ast}=\left\{\begin{array}{l}
\tau_{cp}, \ \    \mbox{if} \ \   0 < \tau_{cp}< \tau_d,\\
\tau_d, \ \ \ \ \mbox{if} \ \ \tau_d \leq \tau_{cp}< \tau_m.
\end{array}\right.
\end{equation}
Then,  \eqref{tauastP3-2} holds too and the theorem  is proved.
 \end{proof}

 If $\tau_{\ast}=\tau_{\Delta}$, then from \eqref{tauDdu} we know that $\|\tau_{\ast}a\|=\Delta$. Therefore, for this case we set $s_{\ast}=\tau_{\ast}a$ and  exit the calculation of subproblem. Otherwise, we know that $\tau_{\ast}a$ is inside the trust region. Then, we should carry out the  calculation of the second stage below.

We set $s=\tau_{\ast}a+y$ and substitute it into $\phi_k (s) $. And then {\color{red}the} subproblem \eqref{phis}-\eqref{sDelta} becomes
\begin{eqnarray}
&&   \min  \ \ \   \displaystyle \psi (y)= \frac{g^T( \tau_{\ast}a+y)}{1- \tau_{\ast}a^Ta}
   + \frac{( \tau_{\ast}a+y)^TB( \tau_{\ast}a+y)}{2(1- \tau_{\ast}a^Ta)^2}, \label{psiy}\\
&& \mbox{s.t.}  \ \ \ \ \  \| y \| \le  \tilde{\Delta}, \ \ a^T y=0, \label{yDa}
\end{eqnarray}
where
\begin{eqnarray}\label{tDelta}
\tilde{\Delta}=\sqrt{\Delta^2-(\tau_{\ast})^2\|a\|^2}.
\end{eqnarray}
In order to remove the   equality constraint in \eqref{yDa}, we  use the null space {\color{red}technique}. That is, for $a\neq 0$ then there exist  $n-1$ mutually orthogonal unit vectors $q, q, \cdots,q_{n-1}$ orthogonal to the parameter vector $a$.
Set $Q=[q, q, \cdots,q_{n-1}]$ and $y=Qu$, where $u \in R^{n-1}$. Then \eqref{psiy}-\eqref{yDa} can be   simplified as following subproblem
\begin{eqnarray}
&&   \min  \ \ \  \tilde{\psi} (u)= \tilde{g}^Tu + \frac{1}{2}u^T\tilde{B}u, \label{tpsiu}\\
&& \mbox{s.t.} \ \ \ \ \  \| u \| \le  \tilde{\Delta},\label{utD}
\end{eqnarray}
where
\begin{eqnarray} \label{tgB}
\tilde{g}= \frac{Q^Tg}{1- \tau_{\ast}a^Ta}+ \frac{\tau_{\ast}Q^TBa}{(1- \tau_{\ast}a^Ta)^2}, \ \
\tilde{B}= \frac{Q^TBQ}{(1- \tau_{\ast}a^Ta)^2}
\end{eqnarray}
Set $g_k=\tilde{g}$, $B_k=\tilde{B}$ and $\Delta_k=\tilde{\Delta}$.
By {\color{red}Algorithm} 1.1, we can obtain the  solution $u_{\ast}$ of the  subproblem \eqref{tpsiu}-\eqref{utD}.
Then $y_{\ast}=Qu_{\ast}$ and  $s_{\ast}=\tau_{\ast}a +y_{\ast}$. Thus, the subproblem \eqref{phis}-\eqref{sDelta} is solved approximately.

Now we could give  the alternating direction search method for solving the conic trust region subproblem \eqref{phis}-\eqref{sDelta} as following.

\

\textbf{ {\color{red}Algorithm} 2.1}

Given $\varepsilon_0, a, g, B$ and $\Delta$.

Step 1.\quad If $a^Tg = 0$, then $\tau_{\ast} = 0$. Set $a=0$ and use {\color{red}Algorithm} 1.1 to get $s_k$, stop.

Step 2.\quad Compute $\tau_{cp},  \tau_d, \tau_u, \tau_{\Delta}$ and $a_{\tau}$ by \eqref{tauDdu}, \eqref{ataum} and \eqref{tau1}.

Step 3.\quad Compute  $1-\Delta \|a\|$.

Step 4.\quad Solve the  subproblem  \eqref{rhotau}-\eqref{tauome}.

\begin{adjustwidth}{1.85cm}{0cm}   Step 4.1.\quad  If $1-\Delta \|a\| \geq \varepsilon_0$, then  calculate $\tau_{\ast}$ by \eqref{tauastP1};
 If $|1- \Delta \|a\| \,  |<  \varepsilon_0$, then  calculate $\tau_{\ast}$ by \eqref{tauastP2};
   \end{adjustwidth}
\begin{adjustwidth}{3.45 cm}{0cm}
  Otherwise, go to step 4.2.
  \end{adjustwidth}
\begin{adjustwidth}{1.85cm}{0cm}
Step 4.2.\quad  If  {\color{red}$a_{\tau}<0$} then  calculate $\tau_{\ast}$ by \eqref{tauastP3-1}; If  {\color{red}$a_{\tau}\geq0$} then  calculate $\tau_{\ast}$ by \eqref{tauastP3-2};
 \end{adjustwidth}

{ \hangafter=1\setlength{\hangindent}{2em}Step 5.\quad If $\tau_{\ast}=\pm\tau_{\Delta}$, then $s_k=\pm\tau_{\Delta}a$, and stop. Otherwise, compute $Q$, $\tilde{\Delta}$, $\tilde{g}$ and $\tilde{B}$ by \eqref{tDelta} and \eqref{tgB}.}

Step 6.\quad Set $g_k=\tilde{g}$, $B_k=\tilde{B}$ and $\Delta_k=\tilde{\Delta}$.
Then solve the subproblem \eqref{tpsiu}-\eqref{utD} by Algorithm 1.1 to get $u_{\ast}$.

Step 7.\quad  Set  $y_{\ast}=Qu_{\ast}$ and $s_k=\tau_{\ast}a +y_{\ast}$, and stop.

\

In order to discuss the lower bound of predicted reduction in each iteration, we define the  following predicted reduction.
\begin{eqnarray}
&& \mbox{pred}(s) = \phi(0) - \phi(s), \ \ \mbox{pred}_1(\tau) = \rho(0) - \rho(\tau) \label{predstau}\\
&& \mbox{pred}_2(y) = \psi(0) - \psi(y), \ \  \mbox{pred}_3(u) = \tilde{\psi}(0) - \tilde{\psi}(u)\label{predyu}
\end{eqnarray}
Now we should prove the following theorem to guarantee the global convergence of the algorithm proposed in the next section.

\begin{theorem}
Under the same conditions as Lemma 2.1. If $s_k = \pm \tau_{\Delta}a$  are obtained by Steps 5
 in  Algorithm 2.1, then there exists a positive constant $c_1$ such that
\begin{equation}\label{pred-skD}
\mbox{pred}({s_ k})\geq \frac{1}{2} c_1 \Delta\| g \|.
\end{equation}
\end{theorem}
\begin{proof}
(1) \ \ If $s_k = \tau_{\Delta}a$, then  we know that $\tau_{\ast}=\tau_{\Delta}$.
By computation, we have
\begin{eqnarray}
&&\mbox{pred}(s_k)=\mbox{pred}_1(\tau_{\Delta}) = -\rho(\tau_{\Delta})\nonumber \\
&&=\frac{-\Delta(\Delta a_{\tau}+2 \| a \|  a^Tg -\Delta \| a \|^2a^Tg)}{2\| a \|^2(1-\Delta \| a \| )^2}, \nonumber
\end{eqnarray}
where   $\tau_{\Delta}$ is generated in two cases as defined in  \eqref{tauastP1} and \eqref{tauastP3-1}.
In both cases, we can find $\tau_{\Delta}\leq\tau_{cp}$ and
\begin{equation}\label{Daa0}
 \Delta a_{\tau}+  \| a \|  a^Tg \leq 0.
\end{equation}
Then
\begin{eqnarray}\label{predtauD1}
\mbox{pred}(s_k)=\mbox{pred}_1(\tau_{\Delta}) \geq\frac{-\Delta a^Tg}{2\| a \|(1-\Delta \| a \| )}.
\end{eqnarray}

 {\color{red}(1a)} \ \ For  $1-\Delta \|a\| \geq \varepsilon_0$, then from \eqref{tauastP1} we know that  $a_{\tau} > 0$  and $ a^Tg<0$.
Combining with \eqref{Daa0} and \eqref{predtauD1} , we have
\begin{eqnarray}\label{predtauD1-2}
 \mbox{pred}(s_k)=\mbox{pred}_1(\tau_{\Delta})\geq\frac{\epsilon\Delta\| g \|}{2},
\end{eqnarray}
where
\begin{eqnarray}\label{epsiag}
\epsilon = \frac{|a^Tg|}{\| a \|\| g \|}.
\end{eqnarray}

 {\color{red}(1b)} \ \ For  $1-\Delta \|a\| \leq -\varepsilon_0$, then from \eqref{tauastP3-1} we know that  $a_{\tau} < 0$  and $ a^Tg>0$.
Because of $1-\Delta \|a\| < 0$ and $a^Tg>0$, then from \eqref{predtauD1} we also have \eqref{predtauD1-2} holds.

(2) \ \ If $s_k = -\tau_{\Delta}a$, then
\begin{eqnarray}\label{pred-tauD1}
&&\mbox{pred}(s_k)=\mbox{pred}_1(-\tau_{\Delta})= -\rho(-\tau_{\Delta}) \nonumber\\
&&=\frac{\Delta(- a_{\tau}\Delta +2 a^Tg \| a \|  + a^Tg\| a \|^2\Delta)}{2\| a \|^2(1+\Delta \| a \| )^2}.
\end{eqnarray}
where  $-\tau_{\Delta}$ is generated in the following  three cases as defined in  \eqref{tauastP1}-\eqref{tauastP3-1} and \eqref{tauastP3-2}.

 {\color{red}(2a)} \ \ For $1-\Delta \|a\| \geq \varepsilon_0$, then $1\leq 1+\Delta \|a\| \leq 2-\varepsilon_0$.

From \eqref{tauastP1}, we know that if $a_{\tau} \leq 0$ then $a^Tg>0$. Thus,
\begin{eqnarray}\label{-atauag0}
- a_{\tau}\Delta + a^Tg \| a \|\geq 0.
\end{eqnarray}
And then, from \eqref{pred-tauD1} we know
\begin{eqnarray}\label{pred-tauD21}
&&\mbox{pred}(s_k)=\mbox{pred}_1(-\tau_{\Delta})  \nonumber\\
&&\geq \frac{\Delta a^Tg }{2\| a \|(1+\Delta \| a \| )} \geq \frac{\epsilon\Delta  \| g \| }{2(2-\varepsilon_0 )}.
\end{eqnarray}

On the other hand, if $a_{\tau} > 0, a^Tg>0$ then $-\tau_{\Delta}\geq\tau_{cp}$.
Then from \eqref{tauDdu} and \eqref{tau1} we have \eqref{-atauag0} holds too. It follows that \eqref{pred-tauD21} holds.

 {\color{red}(2b)} \ \ For  $|1-\Delta \|a\| \, | <  \varepsilon_0$, then $2- \varepsilon_0 < 1+\Delta \|a\| < 2+\varepsilon_0$.

Combining with  \eqref{tauastP2}, we can   prove that  \eqref{-atauag0} holds by the same way and
\begin{eqnarray}\label{pred-tauD22}
&&\mbox{pred}(s_k)=\mbox{pred}_1(-\tau_{\Delta})  \nonumber\\
&&\geq \frac{\Delta a^Tg }{2\| a \|(1+\Delta \| a \| )} \geq \frac{\epsilon\Delta  \| g \| }{2(2+\varepsilon_0 )}.
\end{eqnarray}

 {\color{red}(2c)} \ \ For  $1-\Delta \|a\| \,  \leq  -\varepsilon_0$, then $ 1+\Delta \|a\| \geq 2+\varepsilon_0$.

From \eqref{tauastP3-1}, we know that if $a_{\tau} < 0$, then
\begin{eqnarray*}
\rho (-\tau_{\Delta})\leq \rho (\tau_{\Delta}).
\end{eqnarray*}
By the definition of $\mbox{pred}_1 (\tau)$ in the \eqref{predstau}, we get
\begin{eqnarray*}
\mbox{pred}_1 (-\tau_{\Delta})\geq \mbox{pred}_1 (\tau_{\Delta}).
\end{eqnarray*}
Combining with the proof of the above case  (1a) in this theorem, we have
\begin{eqnarray}\label{pred-tauD2-3}
&&\mbox{pred}(s_k) = \mbox{pred}_1(-\tau_{\Delta}) \nonumber\\
&& \geq \mbox{pred}_1(\tau_{\Delta}) \geq \frac{\epsilon\Delta\| g \|}{2}.
\end{eqnarray}

Therefore, the theorem follows from \eqref{predtauD1-2} and \eqref{pred-tauD21}-\eqref{pred-tauD2-3} with
\begin{eqnarray}\label{c1}
c_1 =\min\left\{\epsilon, \frac{ \epsilon }{ 2-\varepsilon_0}, \frac{ \epsilon }{ 2+\varepsilon_0}\right\}=\frac{ \epsilon }{ 2+\varepsilon_0}.
\end{eqnarray}
\end{proof}

\begin{theorem}
Under the same conditions as Lemma 2.1. If $s_k$ is obtained from the above Algorithm 2.1, then there exists a positive constant $c_4$ such that
\begin{equation}\label{pred-sk}
\mbox{pred}({s_ k})\geq \frac{1}{2} c_4 \| g \|\min\left\{\Delta,\frac{1}{\| a \|}, \frac{\| g \|}{\| B \|}\right\}.
\end{equation}
\end{theorem}
\begin{proof}
(1) \ \ If $s_k$ is obtained by Algorithm 1.1,  
then from  Nocedal and Wright  \cite{Nocedal2006} we have
\begin{equation}\label{predsqua}
\mbox{pred}(s_k) \geq \frac{1}{2} c_2 \| g \| \min\left\{ \Delta, \frac{\| g \|}{\| B \|} \right\},
\end{equation}
where $c_2 \in(0,1]$.

(2) \ \  {\color{red}If $s_k = \pm \tau_{\Delta}a$, then \eqref{pred-skD} holds.}

(3) \ \  $s_k =  \tau_{\ast}a+Qu_{\ast}$, where $\tau_{\ast} \neq \pm \tau_{\Delta}$.
Combining with \eqref{predstau}  and \eqref{predyu}, we have
\begin{equation*}
\mbox{pred}(s_k )=\mbox{pred}_1(\tau_{\ast})+\mbox{pred}_3(u_{\ast}).
\end{equation*}
Because of $u_{\ast}$ is obtained by Algorithm 1.1, then  from   \cite{Nocedal2006} we have
\begin{equation}\label{uast1}
\mbox{pred}_3(u_{\ast}) \geq \frac{1}{2} c_3\| \tilde{g} \| \min\left\{\tilde{\Delta}, \frac{\| \tilde{g} \|}{\| \tilde{B} \|} \right\},
\end{equation}
where $c_3 \in(0,1]$, $\tilde{\Delta}$, $\tilde{g}$  and $\tilde{B}$ as defined by \eqref{tDelta} and \eqref{tgB}.
Thus,
\begin{equation}\label{uast1}
\mbox{pred}(s_k ) \geq \mbox{pred}_1(\tau_{\ast}),
\end{equation}
where $\tau_{\ast}$ can be $\tau_{cp}, \tau_d$ or $\tau_u$.

(3a) \ \  If $\tau_{\ast} = \tau_{cp}$, then  from \eqref{uast1} we have
\begin{eqnarray}\label{predtau1}
&& \mbox{pred}(s_k)\geq   \mbox{pred}_1(\tau_{cp})\nonumber \\
&& = - \frac{\tau_{cp}a^Tg}{ 1-\tau_{cp}\| a \|^2}-\frac{\tau_{cp}^2 a^TBa}{2(1-\tau_{cp}\| a \|^2)^2} \nonumber \\
&& =\frac{(a^Tg)^2}{2a^TBa} = \frac{\epsilon^2\| g \|^2}{2\| B \|},
\end{eqnarray}
where  the second equality is from \eqref{tau1} and the last equality is from \eqref{epsiag}.

(3b) \ \ If  $\tau_{\ast}=\tau_d$, then
\begin{eqnarray*}\label{predtaud}
&& \mbox{pred}(s_k)\geq   \mbox{pred}_1(\tau_d) \nonumber \\
&& = - \frac{\tau_d a^Tg}{ 1-\tau_d\| a \|^2}-\frac{\tau_d^2 a^TBa}{2(1-\tau_d\| a \|^2)^2} \nonumber \\
&& =\frac{-(1-\varepsilon_0)[2\varepsilon_0\| a \|^2a^Tg+ (1-\varepsilon_0)a^TBa]}{2\varepsilon_0^2\| a \|^4}.
\end{eqnarray*}
 From \eqref{tauastP1}-\eqref{tauastP3-1} and \eqref{tauastP3-2},  we know that $\tau_d \leq \tau_{cp}$ and $a^Tg<0$.
For $\tau_d \leq \tau_{cp}$, then we have
\begin{equation*}\label{tiD0}
\varepsilon_0\| a \|^2a^Tg+ (1-\varepsilon_0)a^TBa \leq 0
\end{equation*}
and
\begin{eqnarray}\label{predtaud}
&& \mbox{pred}(s_k)\geq   \mbox{pred}_1(\tau_d) \nonumber \\
&& \geq \frac{-(1-\varepsilon_0)a^Tg }{2 \varepsilon_0\| a \|^2}
= \frac{ \epsilon(1-\varepsilon_0)\|g\| }{2 \varepsilon_0\| a \|},
\end{eqnarray}
where $0<\varepsilon_0<1$.

(3c) \ \ If  $\tau_{\ast}=\tau_u$, then
\begin{eqnarray*}\label{predtauu1}
&& \mbox{pred}(s_k)\geq   \mbox{pred}_1(\tau_u) \nonumber \\
&& = - \frac{\tau_u a^Tg}{ 1-\tau_u\| a \|^2}-\frac{\tau_u^2 a^TBa}{2(1-\tau_u\| a \|^2)^2} \nonumber \\
&& =\frac{(1+\varepsilon_0)[2\varepsilon_0\| a \|^2a^Tg- (1+\varepsilon_0)a^TBa]}{2\varepsilon_0^2\| a \|^4}.
\end{eqnarray*}
 From \eqref{tauastP1}-\eqref{tauastP3-1} and \eqref{tauastP3-2},  we know that $\tau_{cp}\leq \tau_u$, $a_\tau<0$ and $a^Tg>0$.
 For $\tau_{cp}\leq \tau_u$, then we have
\begin{equation*}\label{tiD0}
\varepsilon_0\| a \|^2a^Tg -  (1+\varepsilon_0)a^TBa \geq 0
\end{equation*}
and
\begin{eqnarray}\label{predtauu2}
&& \mbox{pred}(s_k)\geq   \mbox{pred}_1(\tau_u) \nonumber \\
&& \geq \frac{(1+\varepsilon_0)a^Tg }{2 \varepsilon_0\| a \|^2}
= \frac{ \epsilon(1+\varepsilon_0)\|g\| }{2 \varepsilon_0\| a \|}.
\end{eqnarray}

 Therefore, the theorem follows from \eqref{pred-skD}, \eqref{predsqua} and \eqref{predtau1}-\eqref{predtauu2} with
\begin{eqnarray}\label{c1}
c_4 =\min\left\{c_1,c_2,\epsilon^2, \frac{ \epsilon(1-\varepsilon_0) }{ \varepsilon_0}\right\}.
\end{eqnarray}
 \end{proof}

\section{The algorithm and its convergence}
\setcounter{equation}{0}

 In this section, we propose a  quasi-Newton  method with a conic model for unconstrained minimization and prove its convergence under some reasonable conditions. In order to solve the problem \eqref{f}, we approximate $f(x)$ with a conic model of the form
\begin{equation}\label{mk}
m_k(s)=f_k+\frac{ g_k^Ts}{ 1-a_k^Ts}+\frac{1}{2}\frac{ s^TB_ks}{(1-a_k^Ts)^2 },
\end{equation}
where $f_k=f(x_k),\  g_k=\nabla f(x_k)$,
$B_k\in R^{n\times n}$ and
 $a_k \in R^n$ are  parameter vectors.

The choice of the parameters $a_k$ and $B_k$ in \eqref{mk} can refer to \cite{Davidon1980,sc82,so82,zxz95,lu08}  and \cite{Powell1978, Al-Baali2014} respectively.  We set
\begin{eqnarray}
s_{k-1}=x_{k}-x_{k-1},
\end{eqnarray}
\begin{equation}\label{beta}
    \beta=(f_k-f_{k-1})^2-(g_{k-1}^Ts_{k-1})(g_k^Ts_{k-1}),
\end{equation}
If $\beta>0$, then
\begin{equation}\label{betak}
    \beta_k = \frac{f_{k-1}-f_k+\sqrt{\beta}}{-g_{k-1}^Ts_{k-1}};
\end{equation}
otherwise, $\beta_k=1$. In the updating process,  we compute
\begin{equation}\label{ak}
a_k = \frac{1-\beta_k}{g_{k-1}^Ts_{k-1}} g_{k-1},
\end{equation}
\begin{equation}\label{B}
B_{k+1} = B_k-\frac{B_ks_ks_k^TB_k}{s_k^TB_ks_k}+
\frac{z_kz_k^T}{z_k^Ts_k},
\end{equation}
where
\begin{equation}\label{z}
z_k=\theta y_k+(1-\theta)B_ks_k, \  \theta\in[0,1],
\end{equation}
\begin{equation}
\theta=\left\{\begin{array}{l}
1,\qquad \qquad \qquad \mbox{if} \ y_k^Ts_k\geq 0.2s_k^TB_ks_k,\\
\displaystyle  \frac{0.8s_k^TB_ks_k}{s_k^TB_ks_k-y_k^Ts_k}, \;\; \mbox{otherwise},
\end{array}\right.
\end{equation}
and
$y_k=g_{k+1}-g_k$.

Let $s_k$ be the solution of the subproblem \eqref{phis}-\eqref{sDelta}  by Algorithm 2.1. Then either $x_k + s_k$ is accepted as a new iteration point or the trust region radius is reduced according to a comparison between the actual reduction of the objective function
\begin{eqnarray}
\mbox{ared}(s_k)= f(x_k)-f(x_k+s_k)  \label{aredsk}
\end{eqnarray}
and  the reduction predicted by the conic model
\begin{eqnarray}
 \mbox{pred}(s_k)= -\frac{{g_k^{\rm T}}s_k}{1-a_k^Ts_k} - \frac{1}{2}\frac{s_k^TB_ks_k}{(1-a_k^Ts_k)^2}  \label{predsk}
\end{eqnarray}
That is, if the reduction in the objective function is satisfactory, then we finish the current iteration by taking
\begin{eqnarray}
x_{k+1}=x_k+s_k
\end{eqnarray}
and adjusting the trust-region radius; otherwise the iteration is repeated at point $x_k$ with a reduced trust-region radius.

  Now  we give the  alternating direction trust-region algorithm  based on conic model (\ref{mk}).

\

\textbf{Algorithm 3.1 } (\textbf{ADCTR}).

 Step 0.  \ Choose parameters $\epsilon, \varepsilon, \varepsilon_0\in (0,1)$, $0<\eta_1< \eta_2<1$, $0<\delta_1 <1< \delta_2$  and $\bar{\Delta}>0$; give a  starting point $x_0\in R^n$, $B_0 \in R^{n\times n}$, $a_0\in R^n$ and an initial trust region radius $\Delta_0\in(0,\bar{\Delta}]$; set $k=0$.

 Step 1.  Compute $f_k$ and $g_k$. If $\| g_k \|<\varepsilon$, then stop with $x_k$ as the approximate optimal solution; otherwise go to Step 2.

 Step 2. Set $a=a_k$, $g=g_k$, $B=B_k$ and $\Delta=\Delta_k$.
Then solve the subproblem \eqref{phis}-\eqref{sDelta}  by Algorithm 2.1  to get one of the approximate solution $s_k$.

 Step 3. Compute $\mbox{ared}(s_k)$, $\mbox{pred}(s_k)$ and
\begin{equation}\label{rk}
r_k=\frac{\mbox{ared}(s_k)}{\mbox{pred}(s_k)},
\end{equation}

 If $r_k\leq \eta_1$, then set $\Delta_k=\delta_1 \Delta_k$, and go to Step 2.
 If $r_k > \eta_1$, then set $x_{k+1}=x_k+s_k$ and
\begin{equation*}
\Delta_k=\left\{\begin{array}{l}
\min\{\delta_2\Delta_k, \Bar{\Delta}\}, \  \mbox{if} \; r_k\geq \eta_2, \|s_k\|=\Delta_k,\\
\Delta_k, \quad \quad\qquad \ \   \mbox{otherwise}.
\end{array}\right.
\end{equation*}

 Step 4. Generate $a_{k+1}$ and $B_{k+1}$; set $k=k+1$, and go to Step 1.

\  

In this algorithm, the procedure of "Step 2-Step 3-Step 2" is named as inner cycle.
The following theorem guarantees that the ADCTR algorithm does not cycle infinitely in the inner cycle.

\

 \textbf{Assumption 3.1. } \ \ The level set
\[L(x_0)=\{x|f(x)\leq f(x_0)\}\]
and the sequence  $\{\| a_k \|\}$, $\{\| g_k \|\}$ and $\{\| B_k \|\}$ are all uniformly bounded,
 ${B_k}$ is symmetric and positive definite and $f$ is twice continuously differentiable in  $L(x_0)$.

 From \eqref{predsk} and Theorem 2.2, we have
\begin{equation}\label{predskk}
\mbox{pred}({s_ k })\geq \frac{1}{2} c_4\| g_ k \| \min\left\{ \Delta_ k, \frac{1}{\| a_k \|},\frac{\| g_ k \|}{\| B_ k \|} \right\},
\end{equation}
 where $c_1 $ as defined by \eqref{c1}.

\begin{theorem}\quad
Suppose that Assumption 3.1 holds.  $s_k$ is the solution of conic trust-region subproblem
\eqref{phis}-\eqref{sDelta}. If the process does not terminate at $x_k$, then we must have $r_k > \eta_1$ after a finite number of inner iterations.
\end{theorem}
\begin{proof}
We assume that the algorithm does not terminate at $x_k$, then  there is  $\varepsilon_1>0$ such that
 \begin{eqnarray}\label{gkbu0}
\| g_k \|\geq \varepsilon_1.
 \end{eqnarray}
 From  Assumption 3.1 we have
\begin{eqnarray}\label{agBbar}
 \| a_k \| \leq\bar{a},  \ \ \| g_k\|\leq\bar{g}, \ 0<\| B_k\| \leq \bar{B}.
\end{eqnarray}
For simplicity, we suppose that the
superscript denotes the iterative step of inner iteration at $x_k$, then
\begin{eqnarray} \label{rkDeltajk}
r_k \leq \eta_1, \ \Delta^j_{k+1} = \delta_1 \Delta^j_k, \ j = 1, 2, \cdots
\end{eqnarray}
Assume $s^j_k$ is a solution of subproblem \eqref{phis}-\eqref{sDelta} with trust-region radius $\Delta^j_k$, then it is {\color{red}easy} to know that
\begin{eqnarray}\label{Deltajk0}
\lim_{j \rightarrow \infty}  \Delta^j_k=0, \ \lim_{j \rightarrow \infty}  \|s^j_k\| = 0.
\end{eqnarray}
From  \eqref{gkbu0}, \eqref{agBbar} and  \eqref{Deltajk0},   we can obtain that there exist an integer $j_1$ and a constant $\eta_3 > 0$ such that
\begin{equation}\label{predjk}
\mbox{pred}({s^j_ k })\geq \eta_3 \Delta^j_ k, \ \forall j \geq j_1.
\end{equation}
It follows from  \eqref{rkDeltajk} that
\begin{eqnarray} \label{rketa1}
r^j_k = \frac{f_k-f(x_k+s^j_k)}{\mbox{pred}(s^j_k)}\leq  \eta_1.
\end{eqnarray}

On the other hand, from \eqref{Deltajk0} and \eqref{agBbar} we can get
\begin{eqnarray}  \label{1aBksjko}
&& \frac{1}{1-a_k^Ts^j_k}=1+a_k^Ts^j_k +o(\| s^j_k \|),\\
&& \frac{(s^j_k)^TB_ks^j_k}{2(1-a_k^Ts^j_k)^2} = \frac{1}{2}(s^j_k)^TB_ks^j_k +o({\color{red}\| s^j_k \|^2}).
\end{eqnarray}
And then, from \eqref{agBbar}-\eqref{1aBksjko} we have
\begin{eqnarray}
&&\left| f_k-f(x_k+s^j_k) -\mbox{pred}(s^j_k)\right| \nonumber\\
&=&  \left|f_k-f(x_k+s^j_k) +(1+a_k^Ts^j_k)g_k^Ts^j_k +\frac{1}{2}(s^j_k)^TB_ks^j_k+o(\| s^j_k \|^2)\right| \nonumber\\
&=& \left|-\frac{1}{2}(s^j_k)^T\nabla^2 f(x_k+\vartheta_k s^j_k)s^j_k + 
  a_k^Ts^j_k g_k^Ts^j_k+\frac{1}{2}(s^j_k)^TB_ks^j_k+o(\| s^j_k \|^2)\right|  \nonumber \\
&\leq&  \frac{1}{2}( M_1+\bar{B}+2\bar{a}\bar{g}+O(1))\| s^j_k \|^2  \nonumber\\
&\leq&  \frac{1}{2}( Q +O(1))(\Delta^j_k) ^2 , \label{ared-pred2}
\end{eqnarray}
where  $\vartheta_k\in (0,1)$ and $Q = M_1+\bar{B}+2\bar{a}\bar{g}$.
Combining with \eqref{predjk} {\color{red}and \eqref{ared-pred2}}, we can get that
\begin{eqnarray}\label{ffp}
 \left|\frac{f_k-f(x_k+s^j_k)}{\mbox{pred}(s^j_k)}-1\right| \leq  \frac{ ( Q +O(1)) }{ 2\eta_3 }\Delta^j_k , \label{ared-pred}
\end{eqnarray}
holds for all $ j \geq j_1$. By \eqref{Deltajk0} and \eqref{ffp},
\begin{eqnarray}\label{ffpeta}
 \frac{f_k-f(x_k+s^j_k)}{\mbox{pred}(s^j_k)} > \eta_1
\end{eqnarray}
holds for all sufficiently large $j$, which contradicts \eqref{rketa1}. This completes the proof.
\end{proof}

 In the following we give the global convergence property of {\color{red}Algorithm 3.1}.

\begin{theorem}\quad
Suppose that Assumption 3.1 holds.
 Then for any $\varepsilon>0$, the Algorithm 3.1 terminates in finite number of iterations,
 that is
\[\lim_{k\rightarrow\infty} \|g_k\| = 0.\]
\end{theorem}
\begin{proof}
We give the proof by contradiction. Suppose that  there is  $\varepsilon_2>0$ such that
 \begin{eqnarray} \label{av2}
\| g_k \|\geq \varepsilon_2, \ \forall k.
 \end{eqnarray}
Combining with  \eqref{predskk}, \eqref{agBbar} and \eqref{av2}, we have
\begin{eqnarray}\label{pred-sk}
\mbox{pred}({s_ k })\geq  \frac{1}{2} {\color{red}c_4}\varepsilon_2 \min\left\{ \Delta_ k, \frac{1}{\bar{a}},\frac{\varepsilon_2 }{\bar{B}} \right\}
\geq \frac{1}{2}\zeta\Delta_k
\end{eqnarray}
where the first inequality of \eqref{pred-sk} follows from
\begin{eqnarray*}
\min \{p, q, r\} \geq \frac{pqr}{pq +qr + rp },\  \forall p, q, r>0,
\end{eqnarray*}
 and the second inequality is from $\Delta_k \leq \bar{\Delta}$ and
\begin{eqnarray*}
\zeta = \frac{c_1 \varepsilon_2^2}{\varepsilon_2+\bar{B}\bar{\Delta}+\varepsilon_2\bar{a}\bar{\Delta}}.
\end{eqnarray*}
From Steps 3 of Algorithm 3.1 and \eqref{pred-sk}, we obtain that for  {\color{red}all $k$}
\begin{equation}
f_k-f_{k+1}\geq \eta_1 \mbox{pred}(s_k)\geq \frac{1}{2} \eta_1 \zeta\Delta_k.
\end{equation}
Since $f(x)$ is bounded from below and $f_{k+1}<f_k$, we have
\begin{equation}
\infty > \sum _{k\in S} (f_k-f_{k+1})\geq \sum _{k \in S}\left(\frac{1}{2} \eta_1 \zeta\Delta_k\right).
\end{equation}
Combining with Theorem 3.1, we know that
\begin{equation}
\sum _{k =1}^\infty  \Delta_k < \infty,
\end{equation}
which implies that
\begin{eqnarray}\label{Delta=0}
\lim_{ k\rightarrow \infty} \Delta_k =0, \ \lim_{ k\rightarrow \infty} \|s_k\| =0.
\end{eqnarray}

On the other hand,  similar to the proof of  \eqref{1aBksjko}-\eqref{ffpeta} we can obtain
\begin{eqnarray}\label{ffpeta2}
 r_k = \frac{f_k-f(x_k+s_k)}{\mbox{pred}(s_k)} > \eta_1, \ \forall k \geq K,
\end{eqnarray}
where $K$ is sufficiently large.
From Step 3 of Algorithm 3.1, it follows that
\begin{eqnarray*}
\Delta_{k+1}\geq \Delta_k, \ \forall k \geq K,
\end{eqnarray*}
 which is a contradiction to \eqref{Delta=0}.
The theorem is proved.
\end{proof}

\section{Numerical Tests}
\setcounter{equation}{0}
In this section, algorithm  \text{ADCTR} is tested with some  standard test problems from \cite{zxz95, mgh81}.
The purpose of  this paper is to propose a new method to solve the conic trust region subproblem, that is alternating direction method, so we performed algorithm  \text{ADCTR} on a limited number of test problems.
The names of the 16 test problems are listed in Table 1.

All the computations are carried out in Matlab R2015b on a microcomputer in double precision arithmetic.
These tests use the same stopping criterion $\|g_k\|\leq  10^{-5}$. The columns in the Tables have the following meanings:
 \text{No.} denotes the numbers of the test problems; $n$ is the dimension of the test problems; \text{Iter} is the   number of iterations;  $nf$ is the number of function evaluations performed; $ng$ is the number of gradient evaluations; $f_k$ is the final objective function value; $\|g\|$ is the Euclidean norm of the final gradient; CPU(s) denotes the total iteration time of the algorithm in seconds.
 The sign ��*�� means that when the number of iterations reaches 5000, the algorithm fails to stop.
The  parameters in these algorithms are
  $$a_0 = 0, \ B_0 = I, \ \varepsilon_0 = \epsilon= 10^{-5}, \ \Delta_0 = 1, \ \bar{\Delta} = 10, \ \eta_1 = 0.01, \ \eta_2 = 0.75,\  \delta_1 =0.5,\  \delta_2 =2. $$

\begin{table} [!ht]
\centering
\caption{Test functions.}
\begin{tabular}{cccc}
\toprule[1.5 pt]
  No. & Problem       &No.&  Problem  \\
\hline
1  & Cube                   &2 & Penalty-I \\
3  & Beale                  &4 & Conic \\
5  & Extended powell        &6 & Variably Dimensioned \\
7  & Rosenbrock             &8  & Extended Trigonometric  \\
9  & Tridiagonal Exponential  &10 & Brent \\
11 & Troesch                &12 & Cragg and Levy\\
13 & Broyden Tridiagonal    &14 & Brown\\
15& Discrete Boundary Value &16 & Extended Trigonometric  \\
\bottomrule[1.5pt]
\end{tabular}
\end{table}

The numerical results of algorithm  \text{ADCTR} for 16  unconstrained
optimization problems are listed in Table 2.
We note that the optimal value of these test problems is $f_*=0$.
 From Table 2, we can see that our algorithm can obtain the minimum value of the function after a finite number of iterations. And the corresponding minimum point is the stability point, which is also the optimal solution. Therefore, the performance of \text{ADCTR} is   feasible and effective.

\begin{table} [!ht ]
\centering
\caption{Results of ADCTR.}
\begin{tabular}{ccccccc}
\toprule[1.5 pt]
No.& $n$  &  \textit{Iter}&  $nf/ng$ & $f_{k}$& $  \|g\|$   & CPU (s)  \\
  \cline{1-7}
1 & 2 & 52 & 53/43 & 1.0377e-15 & 1.9477e-06 &   0.064681  \\
2 & 2 & 10 & 11/11 & 9.0831e-06 & 8.9419e-06 &   0.048493  \\
3 & 2 & 18 & 19/18 & 9.0379e-15 & 9.3925e-07 &   0.053525  \\
4 & 2 & 16 & 17/13 & 1.1407e-12 & 2.1360e-06 &   0.050445  \\
5 & 4 & 41 & 42/34 & 4.8648e-09 & 4.5887e-06 &   0.062011  \\
6 & 4 & 32 & 33/29 & 2.3856e-14 & 3.0965e-07 &   0.066287  \\
7 & 2 & 50 & 51/49 & 1.5486e-14 & 5.4101e-06 &   0.064227  \\
8 & 4 & 47 & 48/34 & 7.9158e-15 & 4.1153e-07 &   0.076865  \\
9 & 4 & 7 &  8/8 & 8.1577e-12 & 4.5905e-06 &   0.058505  \\
10 & 4 & 81 &  82/58 & 5.8024e-18 & 4.6604e-07 &   0.089702  \\
11 & 4 & 59 &  60/51 & 1.0955e-13 & 2.7230e-06 &   0.077290  \\
12 & 4 & 48 &  49/43 & 1.1247e-08 & 5.2578e-06 &   0.068215  \\
13 & 4 & 35 &  36/19 & 1.4498e-11 & 5.0442e-06 &   0.063276  \\
14 & 2 & 91 &  92/52 & 0.1998e-06 & 2.5916e-07 &   0.089294  \\
15 & 4 & 23 &  24/15 & 2.0042e-12 & 8.2898e-06 &   0.061544  \\
16 & 4 & 14 &  15/15 & 3.0282e-04 & 4.9068e-06 &   0.048488  \\
\bottomrule[1.5pt]
\end{tabular}

\end{table}

In order to analyze the effectiveness of our new algorithm, we compare \text{ADCTR} with the conic quasi-Newton trust region algorithm in which the subproblems are solved by the  dogleg method (\text{DCTR}), see Zhu \cite{zxz95} and Lu \cite{lu08}.  As the dimensions of each test problem ranging from 2 to 4000, we  have actually computed 48 numerical comparisons experiments and the numerical results  are  listed in Table 3.
 Analyzing the numerical results, we have the following conclusions: for the 16 problems, our algorithm \text{ADCTR}  is better
than the \text{DCTR} for 12 tests, is somewhat bad for 2 tests, and the two algorithms are same in efficiency for the other 2
tests; our algorithm in which the subproblems are solved by alternating direction method  is competitive with algorithm \text{DCTR} in \cite{zxz95}. Especially for large-scale problems, our new algorithm has a strong numerical stability.

\begin{table} [!ht ]
\centering
\caption{ Numerical results of \text{DCTR} and \text{ADCTR}}
\begin{tabular}{ccccccc ccc }
\toprule[1.5 pt]
\multicolumn{2}{ c }{Solver} &     \multicolumn{4}{c} {\text{DCTR}}&   \multicolumn{4}{c} {\text{ADCTR}}  \\
\cmidrule(lr){1-2}  \cmidrule(lr){3-6} \cmidrule(lr){7-10}

No. & $  n$   &\textit{Iter} &  $nf/ng$ & $  \|g\|$   & CPU (s)  &\textit{Iter} &  $nf/ng$ &$  \|g\|$   & CPU (s) \\
  \cline{1-10}
 \multirow{3}{*}{1}
     &20&746&747/517& 9.2593e-07&0.116584     &100&101/92&2.1475e-06&0.079679\\
   &200 &2387&2388/2023& 2.6984e-08&2.377265     &82&83/59&4.6090e-06&0.308822\\
    &1000 &*&*/*& *&*      &74&75/56&1.4596e-06&19.36706\\

\multirow{3}{*}{2}
 &200 &76&77/53& 3.1715e-06&0.138719      &79&80/53&4.1422e-06&0.137110\\
  &500 &96&97/62& 6.5292e-06&1.118043   &78&79/54&4.6576e-06&1.270450\\
  &1000 &82&83/57& 5.5181e-06&6.454441   &86&87/57&8.8824e-06&8.026574\\

 \multirow{4}{*}{3}
  &2 &20&21/19& 4.3711e-07&0.042443    &18&19/18&9.3925e-07&0.053525\\
  &20 &24&25/19& 2.4476e-07&0.042636   &24&25/25&4.9198e-06&0.056501\\
  &200 &27&28/24& 3.6411e-08&0.077384    &26&27/27&2.0468e-07&0.145411\\
 &2000 &29&30/25& 6.1805e-06&15.70090    &35&36/28&8.8657e-06&54.47877\\

  \multirow{3}{*}{4}
  &20 &15&16/14& 5.7802e-07&0.039310     &16&17/13&4.8378e-09&0.055520\\
  &200 &16&17/12& 3.7354e-06&0.051864    &19&20/18&3.4444e-07&0.094543\\
 &2000 &18&19/19& 3.9029e-07&13.50780    &19&20/17&2.5377e-06&17.65469\\

\multirow{3}{*}{5}
  &40 &121&122/104& 8.7449e-06&0.064345    &48&49/43& 6.0181e-06&0.076145\\
 &1000 &121&122/116& 2.3550e-06&14.01567    &92&93/77&8.1189e-06&18.098231\\
 &2000 &121&122/116& 2.6445e-06&106.0106    &69&70/60&6.2005e-06&97.24934\\

\multirow{2}{*}{6}
  &40 &120&121/75& 6.0183e-06&0.125199    &145&146/116& 7.2779e-06&0.115281\\
 &400 &*&*/*& *&*       &1124&1125/774&3.2877e-06&11.33673\\

\multirow{3}{*}{7}
  &20 &90&91/69& 1.3138e-06& 0.106181   &83&84/54& 1.0093e-06&0.082832\\
  &200 &517&518/392& 4.5464e-06&0.926231    &61&62/52&2.0797e-06 & 0.242663\\
  &2000 &326&327/294& 2.2237e-06&218.2702     &71&72/54&7.8519e-07&112.6556\\


\multirow{2}{*}{8}
  &4 &46&47/38& 9.2635e-06 & 0.054172        &47&48/34& 4.1153e-07& 0.076865\\
  &40 &*&*/*& *&*            &354&355/265& 9.1963e-06 & 0.147167\\

 \multirow{3}{*}{9}
  &40 &6&7/7& 1.1958e-06 & 0.054117      &6&7/7& 1.8900e-06&0.060258\\
  &400 &6&7/7& 1.6354e-07&0.111561    &6&7/7&2.3374e-07 & 0.133484\\
 &4000 &11&12/12&  8.4467e-07&40.15594   &11&12/12&8.9545e-07&47.93941\\

 \multirow{2}{*}{10}
  &4 &377&378/298& 8.2689e-06 & 0.175454    &81&82/58& 4.6604e-07 &0.089702\\
  &40 &*&*/*& *&*        &1260&1261/910&5.7484e-06&0.391677\\

 \multirow{3}{*}{11}
  &4 &70&71/37& 2.9831e-06 & 0.073789    &59&60/51& 2.7230e-06 &0.077290\\
  &40 &192&193/133& 4.2981e-06 & 0.108448     &132&133/122&3.1390e-06 & 0.116485\\
 &500 &*&*/*& *&*         &1119&1120/1023&9.3082e-06&21.02191\\

 \multirow{3}{*}{12}
  &4 &43&44/41& 4.4263e-06 & 0.062761   &48&49/43& 5.2578e-06 &0.068215\\
  &40 &1977&1978/1315& 8.1245e-06 & 0.369235      &190&191/146&9.5513e-06 & 0.129097\\
 &400 &*&*/*& *&*       &351&352/252&8.4470e-06&4.848008\\

\multirow{4}{*}{13}
  &4 &35&36/16& 8.9785e-06 & 0.053999     &35&36/19& 5.0442e-06&0.063276\\
 &40 &359&360/263& 9.3216e-06&0.140429    &47&48/29&7.5584e-06 & 0.084719\\
 &400 &1996&1997/1400& 9.7095e-06&14.38582    &55&56/34&9.2260e-06&0.746511\\
 &1000 &*&*/*& *&*           &52&53/36&9.5547e-06&10.46032\\

\multirow{3}{*}{14}
  &2 &98&99/59& 5.6830e-06 & 0.058219    &91&92/52& 2.5916e-07&0.089294\\
  &20 &164&165/87& 6.2362e-06&0.076377        &125&126/98&9.9306e-06 & 0.094535\\
  &200 &*&*/*& *&*              &209&210/161&9.3905e-06& 0.656179\\

\multirow{4}{*}{15}
   &4 &27&28/16& 4.2390e-07 & 0.063120     &23&24/15& 8.2898e-06&0.061544\\
 &400 &33&34/11& 8.4956e-06&0.162408       &35&36/15&7.7218e-06 & 0.202917\\ 
 &1000 &21&22/2& 9.0840e-06&0.101637    &21&22/2&9.0840e-06&0.160027\\
&4000 &25&26/2&  5.6751e-07&0.970886    &25&26/2&5.6751e-07&2.331821\\

 \multirow{3}{*}{16}
  &4 &19&20/16& 1.4241e-06 &  0.051494     &14&15/15& 4.9068e-06&0.048488\\
 &40 &518&519/329& 7.5993e-06&0.117250    &63&64/42&6.2298e-06 & 0.057708\\
 &400 &*&*/*& *&*      &60&61/48&4.6296e-06&0.571713\\
\bottomrule[1.5pt]
\end{tabular}

\end{table}

\section{Conclusions}

In this paper, we propose an alternating direction trust region method based on the conic model for
unconstrained optimization and investigate its convergence. Conic models are more flexible to approximate objective functions and have
stronger modeling property.
Alternating direction method (ADM) has been well studied in the context of linearly
constrained convex programming problems.
It is because of the significant efficiency and easy implementation of ADM that
we consider applying it to solving the trust region subproblem based on the conic model.
Initial numerical results show that our new method is competitive  and it {\color{red}is also} effective and robust for  large-scale problems.
The numerical results and the theoretical results lead us to believe that the method is worthy of further study.

In addition, the main purpose of this paper is to   explore a new  method for solving  the conic model subproblem.  Therefore, there are many {\color{red}aspects} worthy of further improvement and research in this paper.
For example, we can consider  the weak convergence assumptions that the Hessian approximations $B_k$ is symmetric
and positive semidefinite. The rate of convergence has not been studied.

\section*{Acknowledgements}
We are grateful to the editors and referees for their suggestions and comments.
This work was supported by National Natural Science Foundation of China (11771210) and  the Natural Science Foundation of Jiangsu Province (BK20141409).

\end{document}